\documentclass[11pt]{amsart}
\usepackage[margin=1in]{geometry}
\usepackage{amsmath}
\usepackage{amsfonts,amsmath}
\usepackage{paralist}
\usepackage[colorlinks=true]{hyperref}
\hypersetup{urlcolor=blue, citecolor=red}
 \numberwithin{equation}{section}
\newtheorem{theorem}{Theorem}[section]

\newtheorem{lemma}[theorem]{Lemma}
\newtheorem{proposition}{Proposition}[section]

\theoremstyle{definition}
\newtheorem{definition}[theorem]{Definition}

\newcommand{\us}{U_\sigma}
\newcommand{\vs}{V_\sigma}

\newcommand*{\avint}{\mathop{\ooalign{$\int$\cr$-$}}}

\newcommand{\io}{\int_{\Omega}}
\newcommand{\ioT}{\int_{\Omega_{T}}}

\newcommand{\ot}{\Omega_T }

\newcommand{\po}{\partial\Omega}

\newcommand{\ep}{\varepsilon}

\newcommand{\mdiv}{\textup{div}}

\newcommand{\pt}{\partial_t}
\newcommand{\nvb}{\|\nabla v\|_{\infty,\ot}}
\newcommand{\nvbn}{\|\nabla v\|_{\infty, \ot}}
\newcommand{\ub}{\| u\|_{\infty,\ot}}
\newcommand{\kn}{K_n}
\newcommand{\kno}{K_{n+1}}

\title[a Keller-Segel model
] 
{ Nonlinear diffusion in the Keller-Segel model of parabolic-parabolic type}

\author[Xiangsheng Xu]{}

\keywords{The Keller-Segel model, chemotaxis equations, non-linear diffusion model, global existence, uniform gradient bounds. {\it J. Differential Equations}, to appear. 
}
\email{xxu@math.msstate.edu}

\begin{document}
	\maketitle
	
	\centerline{\scshape Xiangsheng Xu}
	\medskip
	{\footnotesize
		\centerline{Department of Mathematics \& Statistics}
		\centerline{Mississippi State University}
		\centerline{ Mississippi State, MS 39762, USA}
	} 

	\bigskip

	\begin{abstract}
	In this paper we study the initial boundary value problem for the system $u_t-\Delta u^m=-\mbox{div}(u^{q}\nabla v),\
	v_t-\Delta v+v=u$. This problem is the so-called Keller-Segel model with nonlinear diffusion. Our investigation reveals that nonlinear diffusion can prevent overcrowding. To be precise, we show that solutions are bounded as long as $m>q>0$, thereby substantially generalizing the known results in this area. Furthermore, our result seems to imply that the Keller-Segel model can have bounded solutions and blow-up ones simultaneously.
		
	\end{abstract}
	\section{Introduction}
	Theoretical and mathematical modeling of chemotaxis dates back to the works of Patlak in the 1950s \cite{P} and Keller and Segel in the 1970s \cite{KS}. The general form of the model reads:
	\begin{eqnarray}
	\partial_t u&=&\mbox{div}\left(k_1(u,v)\nabla u-k_2(u,v)\nabla v\right)+k_3(u,v)\ \ \mbox{in $\ot\equiv\Omega\times (0,T]$},\label{eu}\\
	\partial_t v&=&k_c\Delta v+k_4(u,v)\ \ \mbox{in $\ot$},\label{ev}\\
	\frac{\partial u}{\partial \mathbf{n}}&=&	\frac{\partial u}{\partial \mathbf{n}}=0\ \ \mbox{on $\Sigma_T\equiv\po\times(0,T]$},\label{uvb}\\
	(u, v)\mid_{t=0}&=&(u_0(x), v_0(x))\ \ \mbox{on $\Omega$}.\label{uvi}
	\end{eqnarray}
	Here $u$ denotes the cell density and $v$ is the concentration of the chemical signal. The function $k_1$ is the diffusivity of the cells, $k_2$ is the chemotactic sensitivity, $k_3$ describes the cell growth and death. In the signal concentration model, $k_4$ describes the net effect of the production and degradation of the chemical signal. 
As for the remaining terms in the problem,
 $\Omega$ is 
a bounded domain in $\mathbb{R}^N$ with $C^{1,1}$ boundary $\po$, $\mathbf{n}$ the unit outward normal to $\po$, and $T$ any positive number.

Motivated by applications, various assumptions on the given data  were suggested to further simplify the model \cite{HP,S}. In this paper we focus our attention on
	the so-called nonlinear-diffusion model.
	In this case, 
	\begin{equation*}
	k_1=mu^{m-1},\ \ k_2=u^{q},\ \ k_3=0, \ \ k_c=1,\ \  k_4=u-v,
	\end{equation*}
	where
	\begin{equation*}
	m, \ q\in (0,\infty).
	\end{equation*}
The resulting problem is:
\begin{eqnarray}
u_t-\Delta u^m&=&-\mdiv(u^{q}\nabla v)\ \ \mbox{in $\ot$},\label{efu}\\
v_t-\Delta v+v&=&u\ \ \mbox{in $\ot$},\label{efv}\\
	\frac{\partial u^m}{\partial \mathbf{n}}&=&	\frac{\partial v}{\partial \mathbf{n}}=0\ \ \mbox{on $\Sigma_T$},\label{uvb1}\\
(u, v)\mid_{t=0}&=&(u_0(x), v_0(x))\ \ \mbox{on $\Omega$}.\label{uvi1}
\end{eqnarray}

It is certainly beyond the scope of this paper to give a comprehensive review 
 for 
the Keller-Segel model. In this regard, we would like to refer the reader to \cite{H1,H2}. A problem similar to \eqref{efu}-\eqref{uvi1} was investigated in \cite{IY1,IY2,SK} under the assumptions that $N\geq 2, m\geq 1, q\geq 1$ (note that our $q$ here is their $q-1$). The global existence of a weak solution was established if, in addition, $m>q+1-\frac{2}{N}$. When this inequality fails, one obtains local existence and the global existence only holds for small data. H\"{o}lder continuity and uniqueness of weak solutions were considered in \cite{KL}. Some relevance of nonlinear diffusion in chemotaxis was discussed in \cite{CC}.

The objective of this paper is to show that the results in the preceding papers can be substantially improved. Before stating our results, let us define our notion of a weak solution.
\begin{definition}
	We say that $(u, v)$ is a weak solution to \eqref{efu}-\eqref{uvi1} if \begin{eqnarray*}
	u&	\in& L^\infty(\ot),\ u\geq 0,\  u^m\in L^2\left(0,T; W^{1,2}(\Omega)\right),\\
	v&\in& L^\infty(0,T; W^{1, \infty}(\Omega)),\ \ v\geq 0
	\end{eqnarray*}
	and
	\begin{eqnarray*}
	-\ioT u\pt\xi dxdt+\ioT\nabla u^m\nabla\xi dxdt&=&\io u_0\xi(x,0)dx+\ioT u^q\nabla v\nabla \xi dxdt,\\
	-\ioT v\pt\eta dxdt+\ioT\nabla v\nabla\eta dxdt&=&\io v_0\eta(x,0)dx+\ioT (u-v) \eta dxdt
	\end{eqnarray*}
	for each pair of smooth functions $(\xi,\eta)$ with $\xi(x, T)=\eta(x, T)=0$.
\end{definition}
Our main result is:
\begin{theorem}[Main theorem]\label{mth} Assume:
	\begin{enumerate}
		\item[\textup{(H1)}] $\Omega$ is a bounded domain in $\mathbb{R}^N$, $N\geq 3$, with $C^{1,1}$ boundary $\po$;
		\item[\textup{(H2)}] $u_0\in L^\infty(\Omega)$, $v_0\in W^{1,\infty}(\Omega)$ with $u_0\geq 0, v_0\geq 0$;
	\end{enumerate}
	Then there is a weak solution $(u, v)$ to \eqref{efu}-\eqref{uvi1}, provided that one of the following conditions holds.
	\begin{enumerate}
			\item[\textup{(H3)}] $m> q>0$;
		\item[\textup{(H4)}] $m>0,\  1\geq q>0$, and $q+\frac{q-1}{N+1}\leq m\leq q$.
	\end{enumerate}
\end{theorem}
Note that (H4) allows the possibility that $m=q=1$. This is the classical Keller-Segel system, which is well known to have  blow-up solutions. Thus our theorem actually implies that the Keller-Segel model can have bounded solutions and blow-up ones simultaneously. As far as we know, this is the first result in this direction. Our method seems to suggest that solutions blow up as $m\rightarrow q^+$, while solutions remain bounded as $m\rightarrow q^-$ with $q< 1$.
All the results are established under the assumption $N>2$. But it is not difficult to see that Theorem \ref{mth} remains true for $N=2$.

 Motivated by numerical and modeling issues, the question of how blow-up of cells can be avoided has received a lot of attention. One way of doing this is to add a cross-diffusion term to the equation for $v$ \cite{HJ}. A second way is to alter the cell diffusion \cite{BDD}. There are other related works. See, e.g., \cite{CC} in the context of volume effects. Here we show that nonlinear diffusion can also prevent blow-up.

Throughout this paper the letter $c$ is always used to represent a positive number whose value is determined by the given data. The norm of a function in $L^p(\Omega)$ is denoted by $\|\cdot\|_{p,\Omega}$. The Lebesgue measure of a set $D$ in $\mathbb{R}^N$ is represented by $|D|$.  Whenever there is no confusion, we suppress the dependence of a function on its variables, e.g., we write $u$ for $u(x,t)$.
	\section{Preliminaries}
	In this section we collect a few preparatory results.
The first one deals with sequences of non-negative numbers
which satisfy certain recursive inequalities.
\begin{proposition}\label{ynb}
	Let $\{y_n\}, n=0,1,2,\cdots$, be a sequence of positive numbers satisfying the recursive inequalities
	\begin{equation*}
	y_{n+1}\leq cb^ny_n^{1+\alpha}\ \ \mbox{for some $b>1, c, \alpha\in (0,\infty)$.}
	\end{equation*}
	If
	\begin{equation*}
	y_0\leq c^{-\frac{1}{\alpha}}b^{-\frac{1}{\alpha^2}},
	\end{equation*}
	then $\lim_{n\rightarrow\infty}y_n=0$.
\end{proposition}
This proposition can be found in (\cite{D}, p.12).

The following proposition plays a key role in the proof of our main theorem. It can be viewed as a continuous version of Lemma 3.1  in \cite{M,SMG}.
\begin{proposition}\label{prop2.2}
	Let $h(\tau)$ be a continuous non-negative function defined on $[0, T_0]$ for some $T_0>0$. Suppose that there exist three positive numbers $\ep, \delta, b $ such that
	\begin{equation}\label{f2}
	 h(\tau)\leq \ep h^{1+\delta}(\tau)+b\ \ \mbox{for each $\tau \in[0, T_0]$}.
	\end{equation}
Then
\begin{equation}
h(\tau)\leq \frac{1}{[\ep(1+\delta)]^{\frac{1}{\delta}}}\equiv s_0
\ \ \mbox{for each $\tau\in [0, T_0]$},
\end{equation}  provided that
\begin{equation}\label{f1}
\ep\leq \frac{\delta^\delta}{(b+\delta)^\delta(1+\delta)^{1+\delta}}\ \ \mbox{and} \ \ h(0)\leq
 s_0.
\end{equation}
\end{proposition}
\begin{proof}Consider the function $f(s)=\ep s^{1+\delta}-s+b$ on $[0,\infty)$. Then condition \eqref{f2} simply says
	\begin{equation}\label{f3}
	f(h(\tau))\geq 0\ \ \mbox{for each $\tau\in [0,T_0]$.}
	\end{equation} 
It is easy to check that the function $f$ achieves its minimum value 
at $s_0=\frac{1}{[\ep(1+\delta)]^{\frac{1}{\delta}}}$. The minimum value
\begin{eqnarray}
f(s_0)&=&\frac{\ep}{[\ep(1+\delta)]^{\frac{1+\delta}{\delta}}}-\frac{1}{[\ep(1+\delta)]^{\frac{1}{\delta}}}+b\nonumber\\
&=&b-\frac{\delta}{\ep^{\frac{1}{\delta}}(1+\delta)^{\frac{1+\delta}{\delta}}}.
\end{eqnarray}	
By the first inequality in \eqref{f1}, $f(s_0)	\leq -\delta$. Consequently, the equation $f(s)=0$ has exactly two solutions $0<s_1< s_2$ with $s_0$ lying in between.
Evidently, $f$ is positive on $[0, s_1)$, negative on $(s_1,s_2)$, and positive again on $(s_2, \infty)$.
The range of $h$ is a closed interval because of its continuity, and this interval is either contained in $ [0, s_1)$ or $(s_2, \infty)$ due to \eqref{f3}. The latter cannot occur due to the second inequality in \eqref{f1}. Thus the proposition follows.
	\end{proof}
\begin{proposition}\label{nunv} Let $v$ be the solution of the problem
	\begin{eqnarray}
	v_t-\Delta v+v&=&u\ \ \mbox{in $\ot$},\label{v1}\\
		\frac{\partial v}{\partial \mathbf{n}}&=&0\ \ \mbox{on $\Sigma_T$},\label{v2}\\
 v(x,0)&=& v_0(x)\ \ \mbox{on $\Omega$}.\label{v3}
	\end{eqnarray}
	If \textup{(H1)} holds, then for each $p>\frac{N+2}{2}$ there is a positive number $c$ such that
	\begin{equation}\label{fr1}
\sup_{0\leq t\leq T}	\|\nabla v\|_{W^{1,\infty}(\Omega)}\leq c\|\nabla v_0\|_{W^{1,\infty}(\Omega)}+c\|u\|_{2p,\ot}.
	\end{equation}
\end{proposition}
\begin{proof}We do not believe that this result is new. However, we cannot find a good reference to it. So we offer a proof here.
	First we obtain a local interior estimate. The boundary estimate is achieved by flattening the relevant portion of the boundary.
	
	 Now fix a point $z_0=(x_0,t_0)\in \ot$. Then pick a number $R$ from $(0,\min\{\mbox{dist}(x_0,\partial\Omega),\sqrt{t_0}\})$. Define a sequence of cylinders $Q_{R_n}(z_0)$ in $\ot$ as follows:
	\begin{equation*}
	Q_{R_n}(z_0)=B_{R_n}(x_0)\times(t_0-R_n^2, t_0],
	\end{equation*}
	where
	\begin{equation*}
	R_n=\frac{ R}{2}+\frac{R}{2^{n+1}}\ \,\ n=0,1,2,\cdots.
	\end{equation*}
	Choose a sequence of smooth functions $\theta_n$ so that
	\begin{eqnarray*}
		\theta_n(x,t)&=& 1 \ \ \mbox{in $Q_{R_n}(z_0)$},\\
		\theta_n(x,t)&=&0\ \ \mbox{outside $B_{R_{n-1}}(x_0)$ and $t<t_0-R_{n-1}^2$},\\
		|\partial_t\theta_n(x,t)&|\leq &\frac{c4^n}{R^2}\ \ \mbox{on $Q_{R_{n-1}}(z_0)$},\\
		|\nabla \theta_n(x,t)|&\leq & \frac{c2^n}{R}\ \ \mbox{on $Q_{R_{n-1}}(z_0)$,}\ \ \ \mbox{and}\\
		0&\leq &\theta_n(x,t)\leq 1\ \ \mbox{on $Q_{R_{n-1}}(z_0)$.}
	\end{eqnarray*}
	Let $p$ be given as in the lemma.
Select
	\begin{equation}\label{conk1}
	K\geq R^{1-\frac{N+2}{2p}}\|u\|_{2p, Q_R(z_0)}
	\end{equation}
	as below.
	Set
	\begin{equation*}
	K_n=K-\frac{K}{2^{n+1}},\ \ \ n=0,1,2,\cdots.
	\end{equation*}
	Fix an $i\in \{1,\cdot, N\}$. Define
	\begin{equation}\label{coni}
	w=v_{x_i}.
	\end{equation}
Then $w$ satisfies the equation
	\begin{equation}
	w_t-\Delta w+w=u_{x_i}\ \ \mbox{in $\ot$.}\label{w1}
	\end{equation}
	Without loss of generality, assume $\sup_{\ot}w=\|w\|_{\infty,\ot}$.
	We use $\theta_{n+1}^2(w-K_{n+1})^+$ as a test function in \eqref{w1} to derive
	\begin{eqnarray}
	\lefteqn{\frac{1}{2}\frac{d}{dt}\io \theta_{n+1}^2\left[(w-K_{n+1})^+\right]^2dx+\io\theta_{n+1}^2|\nabla(w-K_{n+1})^+|^2dx+\io w\theta_{n+1}^2(w-K_{n+1})^+dx}\nonumber\\
	&=&\io\theta_{n+1}\pt\theta_{n+1}\left[(w-K_{n+1})^+\right]^2dx-2\io\theta_{n+1}\nabla\theta_{n+1}\nabla w(w-K_{n+1})^+dx\nonumber\\
	&&-\io u\theta_{n+1}^2\partial_{x_i}(w-K_{n+1})^+-2\io u\theta_{n+1}\partial_{x_i}\theta_{n+1}(w-K_{n+1})^+dx,
	\end{eqnarray}
	from whence follows
	\begin{eqnarray}
	\lefteqn{\sup_{0\leq t\leq t_0}\io \theta_{n+1}^2\left[(w-K_{n+1})^+\right]^2dx+\int_{0}^{t_0}\io\theta_{n+1}^2|\nabla(w-K_{n+1})^+|^2dxdt}\nonumber\\
	&\leq &\frac{c4^n}{R^2}\int_{Q_{R_n}(z_0)}\left[(w-K_{n+1})^+\right]^2dxdt+c\int_{A_{n+1}} u^2\theta_{n+1}^2dxdt\nonumber\\
		&\leq &\frac{c4^n}{R^2}y_n+c\|u^2\|_{p, Q_R(z_0)}|	A_{n+1}|^{1-\frac{1}{p}},
	\end{eqnarray}
	where
	\begin{eqnarray}
	y_n&=&\int_{Q_{R_n}(z_0)}\left[(w-K_{n})^+\right]^2dxdt,\\
	A_{n+1}&=&\{(x,t)\in Q_{R_n}(z_0): w(x,t)\geq K_{n+1}\}.
	\end{eqnarray}
	By Poincar\'{e}'s inequality,
	\begin{eqnarray}
	\lefteqn{\int_{0}^{t_0}\io \left[\theta_{n+1}(w-K_{n+1})^+\right]^{\frac{4}{N}+2}dxdt}\nonumber\\
	&\leq &\int_{0}^{t_0}\left(\io\left[\theta_{n+1}(w-K_{n+1})^+\right]^2dx\right)^{\frac{2}{N}}\left(\io \left[\theta_{n+1}(w-K_{n+1})^+\right]^{\frac{2N}{N-2}}\right)^{\frac{N-2}{N}}dt\nonumber\\
	&\leq &\left(\sup_{0\leq t\leq t_0}\io\left[\theta_{n+1}(w-K_{n+1})^+\right]^2dx\right)^{\frac{2}{N}}\int_{0}^{t_0}\io \left|\nabla\left(\theta_{n+1}(w-K_{n+1})^+\right)\right|^2dxdt\nonumber\\
	&\leq&c\left(1+4^n\right)\left(\frac{c4^n}{R^2}y_n+c\|u^2\|_{p, Q_R(z_0)}|	A_{n+1}|^{1-\frac{1}{p}}\right)^{\frac{N+2}{N}}.
	\end{eqnarray}
	Subsequently,
	\begin{eqnarray}
	y_{n+1}&=&\int_{Q_{R_{n+1}}(z_0)}\left[(w-K_{n+1})^+\right]^2dxdt\nonumber\\
	&\leq &\int_{0}^{t_0}\io \left[\theta_{n+1}(w-K_{n+1})^+\right]^{2}dxdt\nonumber\\
	&\leq &\left(\int_{0}^{t_0}\io \left[\theta_{n+1}(w-K_{n+1})^+\right]^{2\frac{N+2}{N}}dxdt\right)^{\frac{N}{N+2}}|A_{n+1}|^{\frac{2}{N+2}}\nonumber\\
	&\leq &c4^{\frac{Nn}{N+2}}\left(\frac{c4^n}{R^2}y_n+c\|u^2\|_{p, Q_R(z_0)}|	A_{n+1}|^{1-\frac{1}{p}}\right)|A_{n+1}|^{\frac{2}{N+2}}\nonumber\\
	&\leq &c4^{\frac{Nn}{N+2}}\left(\frac{c4^n}{R^2}y_n+cR^{\frac{N+2}{p}-2}K^{2}|	A_{n+1}|^{1-\frac{1}{p}}\right)|A_{n+1}|^{\frac{2}{N+2}}.\label{r1}
	\end{eqnarray}
	The last step is due to \eqref{conk1}.
	We also have
	\begin{equation}
	y_{n}\geq \int_{A_{n+1}}\left(K_{n+1}-K_n\right)^2dxdt=\frac{K^2}{4^{n+1}}|A_{n+1}|.
	\end{equation}
It immediately follows that
	\begin{eqnarray}
	y_n|A_{n+1}|^{\frac{2}{N+2}}&=&	y_n|A_{n+1}|^{\frac{1}{p}}|A_{n+1}|^{\frac{2}{N+2}-\frac{1}{p}}\nonumber\\
	&\leq&\frac{cR^{\frac{N+2}{p}}4^{\frac{(n+1)[2p-(N+2)]}{p(N+2)}}}{K^{2\frac{2p-(N+2)}{p(N+2)}}} y_n^{1+\frac{2p-(N+2)}{p(N+2)}},\\
K^{2}|	A_{n+1}|^{1-\frac{1}{p}}|A_{n+1}|^{\frac{2}{N+2}}
	&=&K^{2}|	A_{n+1}||A_{n+1}|^{\frac{2}{N+2}-\frac{1}{p}}\nonumber\\
	&\leq&\frac{c4^{n+1+\frac{(n+1)[2p-(N+2)]}{p(N+2)}}}{K^{2\frac{2p-(N+2)}{p(N+2)}}} y_n^{1+\frac{2p-(N+2)}{p(N+2)}}.
	\end{eqnarray}
	Use these in \eqref{r1} to derive
	\begin{eqnarray}
		y_{n+1}&\leq &\frac{cb^nR^{\frac{N+2}{p}-2}}{K^{2\frac{2p-(N+2)}{p(N+2)}}}y_n^{1+\frac{2p-(N+2)}{p(N+2)}}.
	\end{eqnarray}
	By Proposition \ref{ynb}, if we choose $K$ so large that
	\begin{equation}
	y_0\leq cK^{2}R^{N+2},
	\end{equation}
	then
	\begin{equation}
	\sup_{Q_{\frac{R}{2}}(z_0)}w\leq K.
	\end{equation}
	In view of \eqref{conk1}, it is enough for us to take
	\begin{eqnarray}
	K&=&c\left(\frac{y_0}{R^{N+2}}\right)^{\frac{1}{2}}+R^{1-\frac{N+2}{2p}}\|u\|_{2p, Q_R(z_0)}.
	\end{eqnarray}
	Recall that 
	\begin{equation}
	y_0=\int_{Q_{R}(z_0)}\left[\left(w-\frac{K}{2}\right)^+\right]^2dxdt\leq \int_{Q_{R}(z_0)}\left(w^+\right)^2dxdt.
	\end{equation}
	Hence,
	\begin{equation}\label{fr}
	\sup_{Q_{\frac{R}{2}}(z_0)}w\leq c\left(\avint_{Q_{R}(z_0)}\left(w^+\right)^2dxdt\right)^{\frac{1}{2}}+R^{1-\frac{N+2}{2p}}\|u\|_{2p, Q_R(z_0)}.
	\end{equation}
		This is the so-called local interior estimate. Now we proceed to derive the boundary estimate.
	Suppose $x_0\in \po$. Our assumption on the boundary implies that there exist a neighborhood $U(x_0)$ of $x_0$ and a $C^{1,1}$ diffeomorphism $\mathbb{T}$ defined on $U(x_0)$ such that the image of $U(x_0)\cap\Omega$ under $\mathbb{T}$ is the half ball $B_\delta^+(y_0)=\{y: |y-y_0|<\delta, y_i>0\} $, where $\delta>0,  y_0=\mathbb{T}(x_0)$, and $i$ is given as in \eqref{coni}. This implies that we have flatten $U(x_0)\cap\po$ into a region
	in the plane $y_i=0$ in the $y$ space \cite{CFL}.
	Set
	\begin{equation*}
	\tilde{v}=v\circ\mathbb{T}^{-1}, \ \ \tilde{w}=\tilde{v}_{y_i}.
	\end{equation*}
	We can choose $\mathbb{T}$ so that $\tilde{w}=\tilde{w}(y,t)$  satisfies the boundary condition
	\begin{equation}\label{bc}
	\tilde{w}\mid_{y_i=0}=\tilde{v}_{y_i}\mid_{y_i=0}=\frac{\partial\tilde{v}}{\partial\mathbf{n}}\mid_{y_i=0}=0.
	\end{equation}
	One way of doing this is to pick $\mathbb{T}=\left(\begin{array}{c}
	f_1(x)\\
	\vdots\\
	f_N(x)
	\end{array}\right)$ so that the graph of $f_1(x)=0$ is $U(x_0)\cap\po$ and the set of vectors $\{\nabla f_1, \cdots,\nabla f_N\}$ is orthogonal.
	By a result in \cite{X8}, $\tilde{w}$ satisfies the equation
	\begin{equation*}
\pt\tilde{w}	-\mbox{div}\left[(J_{\mathbb{T}}^TJ_{\mathbb{T}})\circ \mathbb{T}^{-1}\nabla\tilde{w}\right]+\tilde{w}=(\mathbf{h}J_{\mathbb{T}})\circ \mathbb{T}^{-1}\nabla\tilde{w}+\left(J_{\mathbb{T}}\circ \mathbb{T}^{-1}\nabla\tilde{u}\right)_{i}\ \ \mbox{in $B_\delta^+(y_0)$},
	\end{equation*}
	where $J_{\mathbb{T}}$ is the Jacobian matrix of $\mathbb{T}$, i.e., 
	\begin{equation*}
	J_{\mathbb{T}}=\nabla\mathbb{T}, 
	\end{equation*} 
	$\left(J_{\mathbb{T}}\circ \mathbb{T}^{-1}\nabla\tilde{u}\right)_{i}$ is the i-th component of the vector $J_{\mathbb{T}}\circ \mathbb{T}^{-1}\nabla\tilde{u}$, and the row vector $\mathbf{h}$ is roughly $\mdiv(J_{\mathbb{T}}^TJ_{\mathbb{T}})$
	 and is, therefore, bounded by our assumption 
	on $\mathbb{T}$. In view of \eqref{bc}, the method employed to prove \eqref{fr} still works here. The only difference is that we use $B_{R_n}^+(y_0)$ instead of $B_{R_n}(y_0)$ in the proof.
	If $t_0=0$, then we just need to change $Q_{R_n}(z_0)$ to $B_{R_n}\times[0, R_n^2)$ and require 
	\begin{equation*}
	K\geq 2\|\nabla v_0\|_{\infty,\Omega},
	\end{equation*}
	in addition to \eqref{conk1} in the proof . Subsequently, \eqref{fr} follows.
	
	Finally, use $v$ as a test function in \eqref{aefv} to derive
	\begin{equation}
\frac{1}{2}	\frac{d}{dt}\io v^2+\io|\nabla v|^2dx+\io v^2dx=\io uv dx.
	\end{equation}
	It immediately follows that
	\begin{equation}
	\ioT|\nabla v|^2dxdt\leq c\ioT u^2dxdt+c\io v_0^2dx.
	\end{equation} 
	The classical $L^\infty$ estimate for linear parabolic equations asserts
	\begin{equation}
	\|v\|_{\infty,\ot}\leq c	\|u\|_{p,\ot}+\|v_0\|_{\infty,\Omega}.
	\end{equation}This completes the proof.


\end{proof}

\section{Proof of Theorem \ref{mth}}
A solution to \eqref{efu}-\eqref{uvi1} is constructed as the limit of a sequence of approximate solutions. Our approximate problems are formulated as follows (also see \cite{SK}):
\begin{eqnarray}
	\pt U-m\mdiv\left((U^++\sigma)^{m-1}\nabla U\right)&=&-\mdiv\left((U^+)^q\nabla V\right)\ \ \mbox{in $\ot$,}\label{cuv1}\\
	\pt V-\Delta V+V&=&U\ \ \mbox{in $\ot$,}\label{cuv2}\\
	\frac{\partial U}{\partial\mathbf{n}}&=&\frac{\partial V}{\partial\mathbf{n}}=0\ \ \mbox{on $\Sigma_T$,}\label{cuv3}\\
	(U,V)\mid_{t=0}&=&(u_0, v_0)\ \ \mbox{on $\Omega$,}\label{cuv4}
\end{eqnarray}
where $\sigma>0$.
The existence of a solution to the above problem can be established via the Leray-Schauder fixed point theorem (\cite{GT}, p.280). To this end, we define an operator $\mathbb{T}$: $L^\infty(\ot)\rightarrow L^\infty(\ot)$ as follows: Let $U\in L^\infty(\ot)$. We say $w=\mathbb{T}(U)$ if $w$ is the unique solution of the problem 
\begin{eqnarray*}
	\pt w-m\mdiv\left((U^++\sigma)^{m-1}\nabla w\right)&=&-\mdiv\left((U^+)^q\nabla V\right)\ \ \mbox{in $\ot$,}\\
	\frac{\partial w}{\partial\mathbf{n}}&=&0\ \ \mbox{on $\Sigma_T$,}\\
	w\mid_{t=0}&=&u_0\ \ \mbox{on $\Omega$,}
\end{eqnarray*}
where $V$ solves the problem
\begin{eqnarray*}
	\pt V-\Delta V+V&=&U\ \ \mbox{in $\ot$,}\\
	\frac{\partial V}{\partial\mathbf{n}}&=&0\ \ \mbox{on $\Sigma_T$,}\\
	V\mid_{t=0}&=&v_0\ \ \mbox{on $\Omega$.}
\end{eqnarray*}
To see that 
$\mathbb{T}$ is well-defined, we conclude from Proposition \ref{nunv} that
$|\nabla V|\in L^\infty(\ot)$. Moreover, the two initial boundary value problems in the definition of $\mathbb{T}$ are both linear and uniformly parabolic. We can infer from (\cite{LSU}, Chap. III) that $w$ is H\"{o}lder continuous in $\overline{\ot}$.  It follows
that $\mathbb{T}$ is continuous and maps bounded sets into precompact ones. We still need to show that there is a positive number c such that 
\begin{equation}\label{uub}
\|U\|_{\infty,\ot}\leq c
\end{equation}
for all $U\in L^\infty(\ot)$ and $\eta\in (0,1)$ satisfying $U=\eta\mathbb{T}(U)$. This equation is equivalent to the following problem
\begin{eqnarray}
\pt U-m\mdiv\left((U^++\sigma)^{m-1}\nabla U\right)&=&-\eta\mdiv\left((U^+)^q\nabla V\right)\ \ \mbox{in $\ot$,}\label{ex1}\\
\pt V-\Delta V+V&=&U\ \ \mbox{in $\ot$,}\\
\frac{\partial U}{\partial\mathbf{n}}&=&\frac{\partial V}{\partial\mathbf{n}}=0\ \ \mbox{on $\Sigma_T$,}\\
(U,V)\mid_{t=0}&=&(\eta u_0, v_0)\ \ \mbox{on $\Omega$.}\label{ex4}
\end{eqnarray}
Use $U^-$ as a test function in \eqref{ex1} to get
\begin{equation*}
-\frac{1}{2}\frac{d}{dt}\io(U^-)^2dx-m\io(U^++\sigma)^{m-1}|\nabla U^-|^2dx=0.
\end{equation*}
Integrate to get
\begin{equation}\label{usp}
U\geq 0\ \ \mbox{a.e. on $\ot$.}
\end{equation}
This implies that 
\begin{equation}\label{vsp}
V\geq 0\ \ \mbox{a.e. on $\ot$.}
\end{equation}
We introduce the following change of dependent variables
\begin{equation}
u=U+\sigma,\ \ v=V+\sigma.
\end{equation}
Then $(u,v)$ satisfies the problem
\begin{eqnarray}
u_t-\Delta u^m&=&-\eta\mdiv(u-\sigma)^{q}\nabla v)\ \ \mbox{in $\ot$},\label{aefu}\\
v_t-\Delta v+v&=&u\ \ \mbox{in $\ot$},\label{aefv}\\
\frac{\partial u}{\partial \mathbf{n}}&=&	\frac{\partial v}{\partial \mathbf{n}}=0\ \ \mbox{on $\Sigma_T$},\label{auvb1}\\
(u, v)\mid_{t=0}&=&(\eta u_0(x)+\sigma, v_0(x)+\sigma)\ \ \mbox{on $\Omega$}.\label{auvi1}
\end{eqnarray}

There is no loss of generality for us to assume that $T\leq 1$. Otherwise, we simply consider $(u(x, Tt), v(x, Tt))$ on $[0,1]$. From here on we will do that, and also let
\begin{equation}
 \sigma\in (0,1).
\end{equation}
We already have $ \eta\in (0,1)$. The generic positive number $c$ will be independent of all three of them.

\begin{lemma}\label{uint} Let \textup{(H3)} hold.
	 Then for each $s$ sufficiently large there is a positive number $c$ such that
	\begin{equation}\label{b12}
	\sup_{0\leq t\leq T}\io u^{s+1}dx+\ioT\left|\nabla u^{\frac{m+s}{2}}\right|^2dxdt
\leq c\nvb^{\frac{m+s}{m-q}}+c.
	\end{equation}
	\end{lemma}
\begin{proof} First remember that 
	\begin{equation}
	\sigma\leq u\in L^\infty(\ot).
	\end{equation}
	Thus for each $r\in \mathbb{R}$, we have
	\begin{equation}
	u^r\in L^2(0,T; W^{1,2}(\Omega)).
	\end{equation}
	Now pick a number 
	\begin{equation}\label{cos1}
	s>\max\{0, m-2q\}.
	\end{equation}
Use $u^s$
as a test function in \eqref{aefu} to derive
\begin{eqnarray}
\lefteqn{\frac{1}{s+1}\frac{d}{dt}\io u^{s+1}dx+ms\io u^{m+s-2}|\nabla u|^2dx}\nonumber\\
&=&s\eta\io (u-\sigma)^{q}u^{s-1}\nabla v\nabla udx\leq s\io u^{q+s-1}|\nabla v\nabla u|dx\nonumber\\
&\leq &\frac{1}{2}ms\io u^{m+s-2}|\nabla u|^2dx+\frac{s\nvb^2}{2m}\io u^{2q-m+s} dx.
\label{t2}
\end{eqnarray}
To estimate the last integral, first notice that
\begin{equation}\label{t3}
\io u^{m+s-2}|\nabla u|^2dx=\frac{4}{(m+s)^2}\io|\nabla u^{\frac{m+s}{2}}|^2dx.
\end{equation}
Recall the Sobolev embedding theorem which states that for each $r\in [1,N)$ there is a positive number $c$ such that
\begin{equation}\label{set}
\|w\|_{\frac{Nr}{N-r}, \Omega}\leq c\|\nabla w\|_{r,\Omega}+c\|w\|_{1,\Omega}\ \ \mbox{for each $w\in W^{1,r}(\Omega)$}.
\end{equation}
We wish to apply this inequality with $w=u^{\frac{m+s}{2}}$ and $r=2$. For this purpose, we further require 
\begin{equation}
\frac{2q-m+s}{m+s}\leq \frac{N}{N-2},
\end{equation}
or equivalently,
\begin{equation}\label{cos2}
s\geq \frac{(2q-m)(N-2)-Nm}{2}.
\end{equation}
We derive from H\"{o}lder's inequality and \eqref{set} that
\begin{eqnarray}
\io u^{2q-m+s}dx&=&\io\left(u^{m+s}\right)^ {\frac{2q-m+s}{m+s}}dx\nonumber\\
&\leq& c\left(\io\left(u^{m+s}\right)^ {\frac{N}{N-2}}dx\right)^{\frac{(2q-m+s)(N-2)}{(m+s)N}}\nonumber\\
&\leq &c\left(\io\left|\nabla u^{\frac{m+s}{2}}\right|^2dx+\left(\io u^{\frac{m+s}{2}} dx\right)^2\right)^{\frac{2q-m+s}{m+s}}\nonumber\\
&\leq & c\left(\io\left|\nabla u^{\frac{m+s}{2}}\right|^2dx\right)^{\frac{2q-m+s}{m+s}}+c\left(\io u^{\frac{m+s}{2}} dx\right)^{\frac{2(2q-m+s)}{m+s}}.\label{s20}
\end{eqnarray}
We integrate \eqref{aefu} over $\Omega$ to get
\begin{equation*}
\frac{d}{dt}\io u dx=0.
\end{equation*}
Subsequently,
\begin{equation}\label{con}
\io u(x,t)dx=\io (\eta u_0(x)+\sigma)dx\leq c\ \ \mbox{for each $t>0$.}
\end{equation}
If we further assume that
\begin{equation}\label{r21}
\frac{m+s}{2}<2q-m+s,
\end{equation}
then we can appeal to the interpolation inequality (\cite{GT}, p.146), thereby deriving
\begin{equation}\label{s30}
\|u\|_{\frac{m+s}{2}, \Omega}\leq \ep\|u\|_{2q-m+s, \Omega}+\frac{1}{\ep^\mu}\|u\|_{1, \Omega}\leq \ep\|u\|_{2q-m+s, \Omega}+\frac{c}{\ep^\mu},
\end{equation}
where $\ep>0,\ \mu=\left(1-\frac{2}{m+s}\right)/\left(\frac{2}{m+s}-\frac{1}{2q-m+s}\right)$.
Condition \eqref{r21} is equivalent to
\begin{equation}\label{cos3}
s>3m-4q.
\end{equation}
Use \eqref{s30} in \eqref{s20} and choose $\ep$ suitably small in the resulting inequality to obtain
\begin{equation}
\io u^{2q-m+s}dx\leq c\left(\io\left|\nabla u^{\frac{m+s}{2}}\right|^2dx\right)^{\frac{2q-m+s}{m+s}}+c.
\end{equation}
Plug this into \eqref{t2} to get
\begin{eqnarray}
\lefteqn{\frac{1}{s+1}\frac{d}{dt}\io u^{s+1}dx+\frac{2ms}{(m+s)^2}\io \left|\nabla u^{\frac{m+s}{2}}\right|^2dx}\nonumber\\
&\leq &c\left(\io\left|\nabla u^{\frac{m+s}{2}}\right|^2dx\right)^{\frac{2q-m+s}{m+s}}\nvb^2+c\nvb^2\nonumber\\
&\leq & \ep \io\left|\nabla u^{\frac{m+s}{2}}\right|^2dx+c(\ep)\nvb^{\frac{m+s}{m-q}}+c\nvb^2.
\end{eqnarray}
The last step is due to the assumption $m>q$ and Young's inequality (\cite{GT}, p. 145). Once again, by taking $\ep$ suitably small, we arrive at
\begin{eqnarray}
\sup_{0\leq t\leq T}\io u^{s+1}dx+\ioT \left|\nabla u^{\frac{m+s}{2}}\right|^2dxdt&\leq& c\nvb^{\frac{m+s}{m-q}}dt+c.
\end{eqnarray}
Here we have used the fact $\frac{m+s}{m-q}>2$ due to \eqref{cos1}. That is to say, the lemma is valid for any $s$ that satisfies \eqref{cos1}, \eqref{cos2}, and \eqref{cos3}. This completes the proof.
\end{proof}

\begin{lemma}\label{h3} Let \textup{(H3)} hold and $s$ be given as in Lemma \ref{uint}. Then  there is a positive number $c$ such that
	\begin{equation}\label{s14}
	\|u\|_{\infty,\ot}\leq c\nvb^{\gamma}+c,
	\end{equation}
	where
	\begin{equation}
	\gamma=\frac{\left[(s+1)(N+2)+N(m-1)^+\right](m+s)+(s+1)N(m-q)(N+2)}{(s+1)(m-q)[(N+2)(s+1)+2N(m-q)]}.
	\end{equation}
	\end{lemma}
\begin{proof}
Let
\begin{equation}\label{s5}
K\geq 2(\|u_0\|_{\infty,\Omega}+1)
\end{equation}
be selected as below.
Define
\begin{eqnarray}
K_n=K-\frac{K}{2^{n+1}}, n=0,1\cdots.
\end{eqnarray}	
Obviously,
\begin{equation}\label{s6}
\frac{K}{2}\leq K_n\leq K.
\end{equation}	
Set
\begin{eqnarray}
S_{n}(t)&=&\{x\in\Omega:u(x,t)\geq \kn\},\label{sndef}\\
A_{n}&=&\cup_{0\leq t\leq T}S_{n}(t)=\{(x,t)\in\ot: u(x,t)\geq \kn\}.\label{andef}
\end{eqnarray}
Subsequently,	
\begin{equation}
\int_{0}^{T}|S_{n+1}(t)|dt=|A_{n+1}|.
\end{equation}
To simplify our presentation, we also introduce two parameters
\begin{eqnarray}
m_s&=&(s+1)\frac{2}{N}+m+s,\label{ps}\\
q_s&=&m_s-(2q-m+s),\label{qs}
\end{eqnarray}
where  $s$ is given as in Lemma \ref{uint}, i.e., $s$ is sufficiently large. Then use $\left(u^s-K_{n+1}^s\right)^+$ as a test function in \eqref{aefu} to derive
\begin{eqnarray}
\lefteqn{\frac{d}{dt}\io \int_{0}^{u}\left(\tau^s-K_{n+1}^s\right)^+d\tau dx+ms\int_{S_{n+1}(t)}u^{m+s-2}|\nabla u|^2dx}\nonumber\\
&=&s\eta\int_{S_{n+1}(t)}(u-\sigma)^qu^{s-1}\nabla v\nabla udx
\leq s\int_{S_{n+1}(t)}u^{q+s-1}|\nabla v||\nabla u|dx.\ \label{s1}
\end{eqnarray}
After a suitable application of Cauchy's inequality (\cite{LSU}, p. 58), we integrate to obtain 
\begin{eqnarray}
\lefteqn{\sup_{0\leq t\leq T}\io \int_{0}^{u}\left(\tau^s-K_{n+1}^s\right)^+d\tau dx+\int_{A_{n+1}}\left|\nabla u^{\frac{m+s}{2}}\right|^2dxdt}\nonumber\\
&\leq& c\int_{A_{n+1}}u^{2q-m+s}|\nabla v|^2dxdt
\leq c\nvb^2\int_{A_{n+1}}u^{2q-m+s}dxdt\label{t1}.
\end{eqnarray}
Since $s>1$, we have
\begin{eqnarray}
\int_{K_{n+1}}^{u}\left(\tau^s-K_{n+1}^s\right)^+d\tau\chi_{A_{K_{n+1}}}
&\geq & \int_{K_{n+1}}^{u}\left[(\tau-K_{n+1})^+\right]^sd\tau\nonumber\\
&=&\frac{1}{s+1}\left[(u-K_{n+1})^+\right]^{s+1}.\label{t10}
\end{eqnarray}
Recall that $m_s=(s+1)\frac{2}{N}+m+s$. We estimate, with the aid of H\"{o}lder's inequality and \eqref{set}, that
\begin{eqnarray}
y_{n+1}&\equiv&\int_{0}^{T}\io\left[(u-K_{n+1})^+\right]^{m_s}dxdt\nonumber\\
&\leq &\int_{0}^{T}\left(\io\left[(u-K_{n+1})^+\right]^{s+1}dx\right)^{\frac{2}{N}}\left(\io\left[(u-K_{n+1})^+\right]^{\frac{(m+s)N}{N-2}}\right)^{\frac{N-2}{N}}dt\nonumber\\
&\leq &c\left(\sup_{0\leq t\leq T}\io\left[(u-K_{n+1})^+\right]^{s+1}dx\right)^{\frac{2}{N}}\nonumber\\
&&\cdot\int_{0}^{T}\left(\io\left|\nabla \left[(u-K_{n+1})^+\right]^{\frac{m+s}{2}}\right|^2dx+\left(\io\left[(u-K_{n+1})^+\right]^{\frac{m+s}{2}}dx\right)^2\right)dt\label{s3}
\end{eqnarray}
We can easily verify that
\begin{eqnarray}
\left|\nabla \left[(u-K_{n+1})^+\right]^{\frac{m+s}{2}}\right|&=&\frac{m+s}{2}\left[(u-K_{n+1})^+\right]^{\frac{m+s}{2}-1}|\nabla u|\nonumber\\
&\leq &\frac{m+s}{2}u^{\frac{m+s}{2}-1}|\nabla u|\chi_{S_{n+1}(t)}=\left|\nabla u^{\frac{m+s}{2}}\right|\chi_{S_{n+1}(t)},\label{m1}\\
\io\left[(u-K_{n+1})^+\right]^{\frac{m+s}{2}}dx
&\leq &\left(\io\left[(u-K_{n+1})^+\right]^{m_s}dx\right)^{\frac{m+s}{2m_s}}|S_{n+1}(t)|^{1-\frac{m+s}{2m_s}}.\label{m2}
\end{eqnarray}
The latter yields
\begin{eqnarray}
\lefteqn{\int_{0}^{T}\left(\io\left[(u-K_{n+1})^+\right]^{\frac{m+s}{2}}dx\right)^2dt}\nonumber\\
&\leq&\int_{0}^{T}\left(\io\left[(u-K_{n+1})^+\right]^{m_s}dx\right)^{\frac{m+s}{m_s}}|S_{n+1}(t)|^{2-\frac{m+s}{m_s}}dt\nonumber\\
&\leq &\left(\ioT\left[(u-K_{n+1})^+\right]^{m_s}dxdt\right)^{\frac{m+s}{m_s}}\left(\int_{0}^{T}|S_{n+1}(t)|^{1+\frac{Nm_s}{2(s+1)}}dt\right)^{\frac{2(s+1)}{Nm_s}}\nonumber\\
&\leq &c |A_{n+1}|^{\frac{2(s+1)}{Nm_s}}y_n^{\frac{m+s}{m_s}}.\label{m3}
\end{eqnarray}
Here we have used the fact that $\{y_n\}$ is a decreasing sequence.
Use \eqref{m1} and \eqref{m3} in \eqref{s3} and take \eqref{t1} into account to derive
\begin{eqnarray}
y_{n+1}
&\leq &c\nvb^{\frac{2(N+2)}{N}}\left(\int_{A_{n+1}}u^{2q-m+s}dxdt\right)^{\frac{N+2}{N}}\nonumber\\
&&+c\nvb^{\frac{4}{N}}\left(\int_{A_{n+1}}u^{2q-m+s}dxdt\right)^{\frac{2}{N}}|A_{n+1}|^{\frac{2(s+1)}{Nm_s}}y_n^{\frac{m+s}{m_s}}.\label{t4}
\end{eqnarray}
The first integral on the right-hand side of \eqref{t4} can be estimated as follows:
\begin{eqnarray}
\left(\int_{A_{n+1}}u^{2q-m+s}dxdt\right)^{\frac{N+2}{N}}
&=&\kno^{\frac{(N+2)(2q-m+s)}{N}}\left(\int_{A_{n+1}}\left(\frac{u}{\kno}\right)^{2q-m+s}dxdt\right)^{\frac{N+2}{N}}\nonumber\\
&\leq &\frac{1}{\kno^{\frac{(N+2)m_s}{N}-\frac{(N+2)(2q-m+s)}{N}}}\left(\int_{A_{n+1}}u^{m_s}dxdt\right)^{\frac{N+2}{N}}\nonumber\\
&= &\frac{1}{\kno^{\frac{(N+2)q_s}{N}}}\left(\int_{A_{n+1}}u^{m_s}dxdt\right)^{1+\frac{2}{N}}.\label{t5}
\end{eqnarray}
Similarly,
\begin{eqnarray}
\left(\int_{A_{n+1}}u^{2q-m+s}dxdt\right)^{\frac{2}{N}}
&\leq & \frac{1}{\kno^{\frac{2q_s}{N}}}\left(\int_{A_{n+1}} u^{m_s}dxdt\right)^{\frac{2}{N}}.\label{t7}
\end{eqnarray}
Recall that
\begin{eqnarray}
K_{n+1}-K_{n}=\frac{K}{2^{n+2}},\ \ \frac{K_{n+1}-K_{n}}{\kno}=\frac{1}{2^{n+2}-1}>\frac{1}{2^{n+2}}.\nonumber
\end{eqnarray}
With the aid of the preceding two results, we obtain 
\begin{eqnarray}
y_n&\geq& \int_{A_{n+1}}\left[(K_{n+1}-K_{n})^+\right]^{m_s}dxdt=\frac{K^{m_s}}{2^{(n+2)m_s}}|A_{n+1}|,\label{t11}\\
y_n&\geq&\int_{A_{n+1}}u^{m_s}\left[\left(1-\frac{K_{n}}{u}\right)^+\right]^{m_s}dxdt\nonumber\\
&\geq &\int_{A_{n+1}}u^{m_s}\left(1-\frac{K_{n}}{K_{n+1}}\right)^{m_s}dxdt
\geq \frac{1}{2^{(n+2)m_s}}\int_{A_{n+1}}u^{m_s}dxdt.\label{t6}
\end{eqnarray}
By \eqref{t11},
\begin{eqnarray}
|A_{n+1}|^{\frac{2(s+1)}{Nm_s}}
&\leq &\frac{2^{\frac{2(n+2)(s+1)}{N}}}{K^{\frac{2(s+1)}{N}}}y_n^{\frac{2(s+1)}{Nm_s}}.\nonumber
\end{eqnarray}
Keeping this, \eqref{t7}, \eqref{s6}, and \eqref{t6} in mind, we derive from \eqref{t4} that
\begin{eqnarray}
y_{n+1}&\leq &\frac{c\nvb^{\frac{2(N+2)}{N}}2^{\frac{(n+2)m_s(N+2)}{N}}}{\kno^{\frac{(N+2)q_s}{N}}}y_n^{1+\frac{2}{N}}\nonumber\\
&&+
\frac{c2^{\frac{\left[2(s+1)+2m_s\right](n+2)}{N}}\nvbn^{\frac{4}{N}}}{K^{\frac{2(s+1)}{N}+\frac{2q_s}{N}}}y_n^{1+\frac{2}{N}}\nonumber\\
&\leq &cb^n\left(\frac{\nvb^{\frac{2(N+2)}{N}}}{K^{\frac{(N+2)q_s}{N}}}+\frac{\nvbn^{\frac{4}{N}}}{K^{\frac{2(s+1)+2q_s}{N}}}\right)y_n^{1+\frac{2}{N}},\label{t12}
\end{eqnarray}
where \begin{equation*}
b=\max\left\{2^{\frac{m_s(N+2)}{N}}, 2^{\frac{2(s+1)+2m_s}{N}}\right\}.
\end{equation*}
We can easily check from \eqref{ps} and \eqref{qs} that
\begin{equation*}
(N+2)q_s\geq 2(s+1)+2q_s\ \ \mbox{if and only if }\ \ m\geq q.
\end{equation*}
Recall that $\kn\geq 1$. Thus if (H3) holds, we can deduce from \eqref{t12} that
\begin{equation*}
y_{n+1}\leq \frac{cb^n\left(\nvb^{\frac{2(N+2)}{N}}+\nvb^{\frac{4}{N}}\right)}{K^{\frac{2(s+1)+2q_s}{N}}}y_n^{1+\frac{2}{N}}.
\end{equation*}
According to Proposition \ref{ynb}, if we choose $K$ so large that
\begin{eqnarray}
y_0&=&\ioT\left[\left(u-\frac{K}{2}\right)^+\right]^{m_s}dxdt
\leq \ioT u^{m_s}dxdt\nonumber\\
&\leq& \frac{cK^{s+1+q_s}}{\nvb^{N+2}+\nvb^{2}}
,\nonumber
\end{eqnarray}
then
\begin{equation*}
\sup_{\ot} u\leq K.
\end{equation*}
In view of \eqref{s5}, it is enough for us to take
\begin{eqnarray}
K&=&c\left(\ioT u^{m_s}dxdt\right)^{\frac{1}{s+1+q_s}}\left(\nvb^{N+2}+\nvb^{2}\right)^{\frac{1}{s+1+q_s}}\nonumber\\
&&+2\|u_0\|_{\infty,\Omega}+2.\label{s11}
\end{eqnarray}
In light of \eqref{s3}, \eqref{m3}, and \eqref{b12}, we have
\begin{eqnarray}
\ioT u^{m_s}dxdt&\leq &\left(\sup_{0\leq t\leq T}\io u^{s+1}dx\right)^{\frac{2}{N}}\ioT\left|\nabla u^{\frac{m+s}{2}}\right|^2dxdt\nonumber\\
&&+c\left(\sup_{0\leq t\leq T}\io u^{s+1}dx\right)^{\frac{2}{N}}\left(\ioT u^{m_s}dxdt\right)^{\frac{m+s}{m_s}}\nonumber\\
&\leq &c\nvb^{\frac{(N+2)(m+s)}{N(m-q)}}+c+\ep\ioT u^{m_s}dxdt+c(\ep)\left(\sup_{0\leq t\leq T}\io u^{s+1}dx\right)^{\frac{m_s}{s+1}}.\nonumber
\end{eqnarray}
Choosing $\ep$ suitably small, we arrive at
\begin{equation*}
\ioT u^{m_s}dxdt\leq c\nvb^{\frac{(N+2)(m+s)}{N(m-q)}}+c\nvb^{\frac{m_s(m+s)}{(s+1)(m-q)}}+c.
\end{equation*}
Substituting this into \eqref{s11} yields
\begin{eqnarray}
\ub&\leq &c\left[\left(\nvb^{\frac{(N+2)(m+s)}{N(m-q)}}+\nvb^{\frac{m_s(m+s)}{(s+1)(m-q)}}+1\right)\left(\nvb^{N+2}+\nvb^{2}\right)\right]^{\frac{1}{s+1+q_s}}+c\nonumber\\
&\leq &c\left[\left(\nvb^{\left(\frac{N+2}{N}+\frac{(m-1)^+}{s+1}\right)\frac{m+s}{m-q}}+1\right)\left(\nvb^{N+2}+1\right)\right]^{\frac{1}{s+1+q_s}}+c\nonumber\\
&\leq &c\left(\nvb^{\frac{\left[(s+1)(N+2)+N(m-1)^+\right](m+s)}{(s+1)N(m-q)}+N+2}+1\right)^{\frac{1}{s+1+q_s}}+c\nonumber\\
&\leq &c\nvb^{\frac{\left[(s+1)(N+2)+N(m-1)^+\right](m+s)+(s+1)N(m-q)(N+2)}{(s+1)N(m-q)(s+1+q_s)}}+c.\nonumber
\end{eqnarray}
This together with \eqref{qs} implies the lemma.
\end{proof}

\begin{proof}[Proof of Theorem \ref{mth} under (H3).] We wish to show
	\begin{equation}\label{mrs}
	\|v\|_{L^{\infty}(0,T; W^{1,\infty}(\Omega))}+\| u\|_{\infty,\Omega\times[0,T]}	\leq c.
	\end{equation}
	Let $\gamma $ be given as in Lemma \ref{h3}. Note that
	\begin{equation*}
	\lim_{s\rightarrow\infty}\gamma=\lim_{s\rightarrow\infty}\frac{\left[(s+1)(N+2)+N(m-1)^+\right](m+s)+(s+1)N(m-q)(N+2)}{(s+1)(m-q)[(N+2)(s+1)+2N(m-q)]}=\frac{1}{m-q}.
	\end{equation*}
If $\frac{1}{m-q}>1$, then there is a $\beta>0$ such that
	\begin{equation}
	\gamma=1+\beta\ \ \mbox{for some suitably large $s$}.
	\end{equation}	Fix this $s$ and let $p$ be given as in Proposition \ref{nunv}. 
We can derive from \eqref{fr1} and Lemma \ref{h3}
	\begin{eqnarray}
	\|v\|_{L^{\infty}(0,T; W^{1,\infty}(\Omega))}+\ub
	&\leq &c\|u\|_{2p,\ot}+c\nvb^{1+\beta}+c
\nonumber\\
&\leq &cT^{\frac{1}{2p}}\|u\|_{\infty,\ot}+c\|u\|_{2p,\ot}^{1+\beta}+c	\nonumber\\
&\leq &cT^{\frac{1}{2p}}\nvb^{1+\beta}+cT^{\frac{1+\beta}{2p}}\ub^{1+\beta}+c\nonumber\\
&\leq &cT^{\frac{1}{2p}}\left(\|v\|_{L^{\infty}(0,T; W^{1,\infty}(\Omega))}+\ub\right)^{1+\beta}+c.\label{s12}
	\end{eqnarray}
	Here we have used the fact that $T\leq 1$.
	Set 
	\begin{equation}
	h(\tau)	=\|v\|_{L^{\infty}(0,\tau; W^{1,\infty}(\Omega))}+\| u\|_{\infty,\Omega\times[0,\tau]}.
	\end{equation}
	Let $T_0\in (0, T]$ be selected below. It follows from \eqref{s12} that
	\begin{equation}
	h(\tau)\leq cT_0^{\frac{1}{2p}}	h^{1+\beta}(\tau)+c\ \ \mbox{for each $\tau\in[0,T_0]$.}
	\end{equation}
It is not difficult for us to see from the proof of Proposition \ref{nunv} that $\nabla v$ is actually H\"{o}lder continuous on $\overline{\ot}$, so is $u$ for each fixed $\sigma>0$. Thus $h(\tau)$ is a continuous function of $\tau$.	In view of Proposition \ref{prop2.2}, if we choose $T_0$ so that
	\begin{equation}\label{s13}
	cT_0^{\frac{1}{2p}}\leq \frac{\beta^\beta}{(c+\beta)^\beta(1+\beta)^{1+\beta}}\ \ \mbox{and}\ \ \|\nabla (v_0+1)\|_{W^{1,\infty}(\Omega)}+\|u_0+1\|_{\infty,\Omega}\leq \frac{1}{\left[cT_0^{\frac{1}{2p}}(1+\beta)\right]^{\frac{1}{\beta}}}
	\end{equation}
	then
	\begin{equation}
	\|\nabla v\|_{\infty,\Omega\times[0,T_0]}\leq \frac{1}{\left[cT_0^{\frac{1}{2p}}(1+\beta)\right]^{\frac{1}{\beta}}}.
	\end{equation}
	By setting $T=0$ in \eqref{s12}, we see that $\|\nabla (v_0+1)\|_{W^{1,\infty}(\Omega)}+\|u_0+1\|_{\infty,\Omega}\leq c$.
	If we take
	\begin{equation}
	cT_0^{\frac{1}{2p}}=\frac{\beta^\beta}{(c+\beta)^\beta(1+\beta)^{1+\beta}},
	\end{equation}
	then the second inequality in \eqref{s13} is automatically satisfied. Upon doing so, we arrive at
	\begin{equation}\label{s15}
\|v\|_{L^{\infty}(0,T_0; W^{1,\infty}(\Omega))}+\| u\|_{\infty,\Omega\times[0,T_0]}	\leq\frac{(c+\beta)(1+\beta)}{\beta}.
	\end{equation}
	Set $k=\lfloor\frac{T}{T_0}\rfloor$, the integer part of the number $\frac{T}{T_0}$. 
	If $k\geq 1$, we consider 
	\begin{equation}
	u_{T_0}(x,t)=u(t+T_0,x),\ \ 	v_{T_0}(x,t)=v(t+T_0,x) \ \ \mbox{on $[0, T_0]$.}
	\end{equation} 
	Obviously, $(u_{T_0}, v_{T_0})$ satisfies the same conditions as $(u, v)$ does on $\Omega\times(0,T_0)$. Thus we can repeat the previous arguments  to yield \eqref{s15} for $(u_{T_0}, v_{T_0})$. After a finite number of steps, we obtain \eqref{mrs}. Of course, in the last step, we will have to use $\min\{T_0, T-kT_0\}$ instead of $T_0$.
	
	If $\frac{1}{m-q}<1$, an application of Young's inequality is enough to reach \eqref{mrs}.

	If $\frac{1}{m-q}=1$, this can also be handled easily.  We verify that $\frac{d\gamma}{ds}$ changes signs at most three times. Thus either $\gamma $ decreases toward 1 as $s\rightarrow \infty$, which can be treated like the first case, or  $\gamma $ increases toward 1 as $s\rightarrow \infty$, which is essentially the second case.
	
Clearly, \eqref{uub} is a consequence of \eqref{mrs}. Thus we can conclude from the Leray-Schauder fixed point theorem that  \eqref{cuv1}-\eqref{cuv4} has a solution. Denote the solution by $(\us, \vs)$. In view of \eqref{usp}, we can rewrite \eqref{cuv1}-\eqref{cuv4} as
	\begin{eqnarray}
	\pt \us-m\mdiv\left((\us+\sigma)^{m-1}\nabla \us\right)&=&-\mdiv\left(\us^q\nabla \vs\right)\ \ \mbox{in $\ot$,}\label{acuv1}\\
	\pt \vs-\Delta \vs+\vs&=&\us\ \ \mbox{in $\ot$,}\label{acuv2}\\
	\frac{\partial \us}{\partial\mathbf{n}}&=&\frac{\partial \vs}{\partial\mathbf{n}}=0\ \ \mbox{on $\Sigma_T$,}\label{acuv3}\\
	(\us,\vs)\mid_{t=0}&=&(u_0, v_0)\ \ \mbox{on $\Omega$.}\label{acuv4}
	\end{eqnarray}
Furthermore,
	\begin{eqnarray}
	\us\geq 0,\ \ \vs\geq 0,\ \ \mbox{and}\ \ \|\vs\|_{L^\infty(0,T; W^{1,\infty}(\Omega))}+\|\us\|_{\infty,\ot}\leq c.\nonumber
	\end{eqnarray}
	We wish to show that we can take $\sigma \rightarrow 0$ in \eqref{acuv1}-\eqref{acuv4}.
	 For this purpose, we use $(\us+\sigma)^m$ as a test function in \eqref{acuv1} to derive
	 \begin{equation}
	 \frac{1}{m+1}\sup_{0\leq t\leq T}\io(\us+\sigma)^{m+1}dx+\ioT\left|\nabla(\us+\sigma)^m\right|^2dxdt\leq c.
	 \end{equation}
	 We compute
	 \begin{eqnarray}
	 \pt (\us+\sigma)^{m+1}&=&(m+1)(\us+\sigma)^{m}\pt\us\nonumber\\
	 &=&(m+1)\mdiv\left((\us+\sigma)^{m}\nabla(\us+\sigma)^{m}\right)-(m+1)\left|\nabla(\us+\sigma)^{m}\right|^2\nonumber\\&&-(m+1)\mdiv\left((\us+\sigma)^{m}\us^q\nabla \vs\right)+(m+1)\us^q\nabla \vs\cdot\nabla(\us+\sigma)^{m},\nonumber\\
	 \nabla(\us+\sigma)^{m+1}&=&\frac{m+1}{m}(\us+\sigma)\nabla(\us+\sigma)^{m}.\nonumber
	 \end{eqnarray}
	 Thus the sequence $\{\partial_t(\us+\sigma)^{m+1}\}$ is bounded in $L^2\left(0, T; \left(W^{1,2}(\Omega)\right)^*\right)+L^1(\ot)\equiv\{\psi_1+\psi_2:\psi_1\in L^2\left(0, T; \left(W^{1,2}(\Omega)\right)^*\right), \psi_2\in L^1(\ot) \}$ and the sequence $\{(\us+\sigma)^{m+1}\}$ is bounded in $L^2\left(0, T; W^{1,2}(\Omega)\right)$. This puts us in a position to apply the Lions-Aubin lemma \cite{SI}. Upon doing so, we obtain the precompactness of $\{(\us+\sigma)^{m+1}\}$ in $L^2(\ot)$. We can extract a subsequence of $\{\us+\sigma\}$, still denoted by $\{\us+\sigma\}$, such that $\us+\sigma$ converges a.e. on $\ot$. This is enough to justify passing to the limit in \eqref{acuv1}-\eqref{acuv4}.
	 The proof is complete.
	\end{proof}

We would like to remark that as $m\rightarrow q^+$ the upper bound in \eqref{s15} deteriorates. This foretells the possibility that solutions blow up if $m=q$. 
	
\begin{proof}[Proof of Theorem \ref{mth} under (H4).] We will show that an estimate like \eqref{s14} remains true even without the benefit of Lemma \ref{uint}. Let $s$ be given as before, i.e., $s$ is large enough.
With the aid of (H4), we can derive from \eqref{t12} that
\begin{eqnarray}
y_{n+1}&\leq &\frac{cb^n\left(\nvb^{\frac{2(N+2)}{N}}+\nvb^{\frac{4}{N}}\right)}{K^{\frac{(N+2)q_s}{N}}}y_n^{1+\frac{2}{N}}.\nonumber
\end{eqnarray}
In light of Proposition \ref{ynb}, if $K$ is so chosen that
\begin{eqnarray}
y_0
&\leq &\frac{cK^{\frac{(N+2)q_s}{2}}}{\nvb^{N+2}+\nvb^{2}},\nonumber
\end{eqnarray}
then
\begin{equation}\label{s17}
\sup_{\ot}u\leq K.
\end{equation}
In view of \eqref{s5}, it is enough for us to take
\begin{eqnarray}
K&=&c\left(\ioT u^{m_s}dxdt\right)^{\frac{2}{(N+2)q_s}}\left(\nvb^{N+2}+\nvb^{2}\right)^{\frac{2}{(N+2)q_s}}\nonumber\\
&&+2\|u_0\|_{\infty,\Omega}+2.\label{s18}
\end{eqnarray}
If 
\begin{equation}
\frac{2m_s}{(N+2)q_s}<1,
\end{equation}
or equivalently,
\begin{equation*}
 q<1 \ \ \mbox{and}\ \  m> q+\frac{q-1}{N+1},
\end{equation*}
then Young's inequality asserts
\begin{eqnarray}
K&\leq &\ep \|u\|_{m_s,\ot} +c(\ep)\left(\nvb^{N+2}+\nvb^{2}\right)^{\frac{2}{(N+2)q_s-2m_s}}\nonumber\\
&&+2\|u_0\|_{\infty,\Omega}+2.\nonumber
\end{eqnarray}
Use this in \eqref{s17} to derive
\begin{equation}\label{q1}
\ub\leq c\nvb^{\frac{N+2}{(N+1)m-(N+2)q+1}}+c\|u_0\|_{\infty,\Omega}+c.
\end{equation}
If
\begin{equation}
\frac{2m_s}{(N+2)q_s}= 1,
\end{equation}
 we can appeal to the interpolation inequality (\cite{GT}, p. 146) to obtain
\begin{equation}
\|u\|_{m_s,\ot}\leq \ep \|u\|_{\infty,\ot}+\frac{1}{\ep^{m_s-1}}\|u\|_{1,\ot}\leq\ep \|u\|_{\infty,\ot}+\frac{c}{\ep^{m_s-1}}.
\end{equation}
With this in mind, we derive from \eqref{s18} that
\begin{eqnarray}
K&\leq&c\left(\ep \|u\|_{\infty,\ot}+\frac{c}{\ep^{m_s-1}}\right)\left(\nvb^{N+2}+\nvb^{2}\right)^{\frac{2}{(N+2)q_s}}\nonumber\\
&&+2\|u_0\|_{\infty,\Omega}+2\nonumber\\
&=&\alpha\ub+\frac{c}{\alpha^{m_s-1}}\left(\nvb^{N+2}+\nvb^{2}\right)^{\frac{2m_s}{(N+2)q_s}}\nonumber\\
&&+2\|u_0\|_{\infty,\Omega}+2.\nonumber
\end{eqnarray}
Plug this into \eqref{s17} and choose $\alpha$ suitably small in the resulting inequality to derive
\begin{equation}\label{q2}
\ub\leq c\nvb^{N+2}+c\|u_0\|_{\infty,\Omega}+c.
\end{equation}
The rest of the proof is similar to that under (H3). That is, \eqref{mrs} can be inferred from either \eqref{q1} or \eqref{q2}.
\end{proof}
	
\end{document}